\documentclass[a4paper,reqno,11pt]{amsart}
\usepackage[a4paper,margin=2.5cm]{geometry}
\usepackage[T1]{fontenc}
\usepackage[utf8]{inputenc}
\usepackage[english]{babel}
\usepackage{amssymb,amsmath,amsthm,mathtools}
\usepackage[dvipsnames,svgnames,table]{xcolor}
\usepackage{graphicx}
\usepackage{caption,subcaption}
\usepackage{lmodern}
\usepackage[foot]{amsaddr}
\usepackage[shortlabels]{enumitem}
\usepackage[unicode=true]{hyperref}
\usepackage[capitalise,noabbrev]{cleveref}
\usepackage{thmtools,thm-restate}
\usepackage{comment}
\usepackage{mdframed}
\usepackage{algorithmicx}
\usepackage[noend]{algpseudocode}

\newcommand{\q}[1]{``#1''}
\newcommand{\defin}[1]{\emph{\textcolor{ForestGreen}{#1}}}
\newcommand{\defmath}[1]{\defin{\text{$#1$}}}

\newcommand{\calC}{\mathcal{C}}

\DeclareMathOperator{\dist}{dist}

\DeclareMathOperator{\depth}{depth}

\DeclareMathOperator{\pw}{pw}

\let\le\leqslant
\let\ge\geqslant
\let\leq\leqslant
\let\geq\geqslant

\let\subset\subseteq

\let\epsilon\varepsilon

\let\setminus\backslash

\crefformat{equation}{#2(#1)#3}
\crefformat{subsection}{Subsection #2#1#3}

\makeatletter
\def\thm@space@setup{
  \thm@preskip=3mm
  \thm@postskip=0mm
}
\makeatother

\algdef{SE}[DOWHILE]{Do}{doWhile}{\algorithmicdo}[1]{\algorithmicwhile\ #1}

\mdfdefinestyle{dontsplit}{
  hidealllines=true,
  nobreak=true,
  leftmargin=0pt,
  rightmargin=0pt,
  innerleftmargin=0pt,
  innerrightmargin=0pt,
}

\newmdtheoremenv[style=dontsplit]{theorem}{Theorem}
\newmdtheoremenv[style=dontsplit]{lemma}[theorem]{Lemma}
\newmdtheoremenv[style=dontsplit]{observation}[theorem]{Observation}
\newmdtheoremenv[style=dontsplit]{proposition}[theorem]{Proposition}
\newmdtheoremenv[style=dontsplit]{question}[theorem]{Question} 
\newmdtheoremenv[style=dontsplit]{corollary}[theorem]{Corollary} 
\newmdtheoremenv[style=dontsplit]{problem}[theorem]{Problem}
\newmdtheoremenv[style=dontsplit]{conjecture}[theorem]{Conjecture}

\newtheorem*{problem*}{Problem}
\newtheorem*{conjecture*}{Conjecture}
\newtheorem*{remark*}{Remark}
\newtheorem*{question*}{Question} 

\theoremstyle{remark}
\newmdtheoremenv[style=dontsplit]{claim}[theorem]{Claim}

\crefname{claim}{Claim}{Claims}

\newtheorem*{claim*}{Claim}

\graphicspath{{figs/}}

\setlist[enumerate,1]{label=\textup{(\roman*)}}

\makeatletter
\newcommand{\myitem}[1]{%
\item[#1]\protected@edef\@currentlabel{#1}%
}
\makeatother

\hypersetup{
    colorlinks,
    linkcolor={RoyalBlue},
    citecolor={RubineRed},
    urlcolor={blue!80!black},
    pdftitle={Row pathwidth of complete binary trees}
}

\DeclareMathOperator{\rpw}{rpw}

\begin{document}

\title[Row pathwidth of complete binary trees]
{Row pathwidth of complete binary trees}

\author[J.~Hodor]{J\k{e}drzej Hodor}
\address[J.~Hodor]{Theoretical Computer Science Department,
Faculty of Mathematics and Computer Science, Doctoral School of Exact and Natural Sciences,
Jagiellonian University, Krak\'ow, Poland}
\email{\href{mailto:jedrzej.hodor@gmail.com}{jedrzej.hodor@gmail.com}}

\author[P.~Micek]{Piotr Micek}
\address[P.~Micek]{Theoretical Computer Science Department, Faculty of Mathematics and Computer Science, Jagiellonian University, Kraków, Poland.}
\email{\href{mailto:piotr.micek@uj.edu.pl}{piotr.micek@uj.edu.pl}}
\thanks{The authors were partially supported by the National Science Center, Poland,  under grant UMO-2023/05/Y/ST6/00079 within the WEAVE-UNISONO program.}
\begin{abstract}
We show that if a complete binary tree of height $h$ is isomorphic to a subgraph of the strong product of a graph $H$ and a path, then $\pw(H)$ is $\Omega(h)$.
This solves a problem posed by Bose, Dujmović, Javarsineh, Morin, and Wood (2022). 
The proof was found by OpenAI's GPT-5.6 Sol Pro.
\end{abstract}

\maketitle

\section{Introduction}\label{sec:introduction}

%\piotr{Write somewhere that we do not optimize the multiplicative constant.}

%For graphs $A$ and $B$, let \defin{$A\boxtimes B$} denote their strong product.
The \defin{row pathwidth} of a graph $G$, denoted \defin{$\rpw(G)$}, is the minimum pathwidth of a graph $H$ such that $G$ is isomorphic to a subgraph of the strong product of $H$ and a path.
Pathwidth and the strong product are defined in \cref{sec:proof}.
Let \defin{$T_h$} denote the complete binary tree of height~$h$, i.e., the binary tree with $2^h$ leaves and all root-to-leaf paths of length $h$.
The main result of this paper is the following.

\begin{theorem}\label{thm:main}
For every nonnegative integer $h$,
\[
    \left\lfloor(h+1)/16\right\rfloor
    \le \rpw(T_h).
\]
\end{theorem}
We make no effort to optimize the constant in~\cref{thm:main}.

We first give some brief context for our result. 
A central goal of the \emph{graph product structure theory} is to represent graphs as subgraphs of strong products of simpler graphs.
The theory was initiated by Dujmovi\'{c}, Joret, Micek, Morin, Ueckerdt, and Wood~\cite{DJMMUW20}, who proved that every planar graph is isomorphic to a subgraph of the strong product of a graph of treewidth at most $8$ and a path; in other words, planar graphs have \q{row treewidth} at most $8$.
This led to a broad range of product structure theorems and applications throughout structural and algorithmic graph theory.

%In this paper, we study the gap between row pathwidth and pathwidth of graphs. 
Complete binary trees are canonical obstructions to small pathwidth, i.e.\ $\pw(T_h)=\left\lceil h/2\right\rceil$ and as first shown in \emph{Graph Minors I} by Robertson and Seymour~\cite{GM1}, 
every graph of large pathwidth contains a subdivision of a large complete binary tree. 
What is the row pathwidth of a complete binary tree $T_h$? 
A result of Dvo\v{r}\'ak, Huynh, Joret, Liu, and Wood~\cite[Theorem~13]{DHJLW21} implies a lower bound $\Omega(h/\log h)$, while $\rpw(T_h)\le \pw(T_h)=\left\lceil h/2\right\rceil$. 
Determining the asymptotic order of $\rpw(T_h)$ was posed as an open problem by Bose, Dujmovi\'c, Javarsineh, Morin, and Wood~\cite{BDJMW22}. 
%A result of Dvo\v{r}\'ak, Huynh, Joret, Liu, and Wood~\cite[Theorem~13]{DHJLW21} implies the lower bound $\Omega(h/\log h)$, while $\rpw(T_h)\le \pw(T_h)=\left\lceil h/2\right\rceil$.
\cref{thm:main} gives $\rpw(T_h) \in \Theta(h)$.

%Here, we study, a somewhat negative side of the theory.
%In the pathwidth setting, Bose et al.~\cite{BDJMW22} proved that for every positive integer $k$, there exists a tree whose pathwidth and row pathwidth are both equal to $k$; their construction uses trees with sufficiently large branching.
%They then asked whether pathwidth and row pathwidth are also of the same order for complete binary trees.
%Equivalently, they asked whether the row pathwidth of a complete binary tree is linear in its height.
%We answer this question affirmatively.

\section{Proof}\label{sec:proof}
For a positive integer $n$, we write $\defmath{[n]} = \{1,\dots,n\}$. 

The \defin{strong product} of graphs~$A$ and~$B$, denoted by \defin{${A \boxtimes B}$}, is the graph with vertex-set ${V(A) \times V(B)}$, where distinct vertices ${(v,x),(w,y) \in V(A) \times V(B)}$ are adjacent if
${v=w}$ and ${xy \in E(B)}$, or
${x=y}$ and ${vw \in E(A)}$, or
${vw \in E(A)}$ and~${xy \in E(B)}$.

Let $H$ be a graph.
For two vertices $u$ and $v$ in $H$, the \defin{distance} from $u$ to $v$ in $H$, denoted \defin{$\dist_H(u,v)$}, is the minimum length of a path from $u$ to $v$ in $H$.
For a vertex $z\in V(H)$ and a nonnegative integer $d$,  we define the \defin{ball} of \defin{radius} $d$ \defin{centered} at $z$ in $H$ by $\defmath{B_H(z,d)}=\{x\in V(H):\dist_H(x,z)\le d\}$.
A sequence $(B_1,\dots,B_m)$ of subsets of the vertices of $H$ is a \defin{path decomposition} of $H$ if 
\begin{enumerate}
    \item for every $v \in V(H)$, the set $\{i \in [m] : v \in B_i\}$ is a nonempty interval in $[m]$, and
    \item for every $uv \in E(H)$, there exists $i \in [m]$ with $u,v \in B_i$.
\end{enumerate}
The \defin{width} of this path decomposition is $\max_{i \in [m]}|B_i| - 1$.
The \defin{pathwidth} of $H$, denoted by \defin{$\pw(H)$}, is the minimum width of a path decomposition of $H$.

We start with the following folklore observation:  
given a connected graph $H$, a path decomposition $(B_1,\ldots,B_m)$ of $H$ of width at most $k$, and a path $P$ in $H$ 
meeting both $B_1$ and $B_m$, 
we have that $\pw(H-V(P))<k$. 
This observation yields the following lemma.

\begin{lemma}\label{lem:separator}
Let $k$ be a positive integer, let $H$ be a connected graph with $\pw(H)\le k$, and let $z\in V(H)$.
Then there is a set $Y\subseteq V(H)$ such that
\begin{enumerate}
    \item $z \in Y$,
    \item $\pw(H-Y)\le k-1$, and
    \item $|Y \cap B_H(z,d)| \leq 2d+1$ 
    for every nonnegative integer $d$. 
\end{enumerate}
\end{lemma}

\begin{proof}
Let $(B_1,\ldots,B_m)$ be a path decomposition of $H$ of width at most $k$.
Without loss of generality, we may assume that $B_1$ and $B_m$ are nonempty.
Let $Q_1$ be a shortest path in $H$ from $z$ to a vertex of $B_1$. 
Let $Q_m$ be a shortest path in $H$ from $z$ to a vertex of $B_m$. 
(Both paths exist since $H$ is connected.) 
Let $Y=V(Q_1)\cup V(Q_m)$. 
Clearly, $z\in Y$. 
In any path decomposition, 
the set of indices of bags meeting a connected subgraph is an interval.
Since $H[Y]$ is connected and meets both $B_1$ and $B_m$, every bag meets $Y$.
Thus, deleting $Y$ from every bag gives a path decomposition of $H-Y$ with bags of size at most $k$.  This proves $\pw(H-Y)\le k-1$.
Each of the two shortest paths contains at most $d+1$ vertices at distance at most $d$ from $z$, and both contain $z$.
Hence $|Y\cap B_H(z,d)|\le 2d+1$.
This completes the proof.
\end{proof}

A \defin{rooted tree} is a tree with a distinguished vertex, called its \defin{root}. 
For our purposes, a rooted tree is \defin{binary} if every vertex has at most two children.
For a rooted tree $T$ with root $r$, define $\defmath{\depth_T(v)}=\dist_T(r,v)$.
Let $q = 3 / 4$.
For each rooted tree $T$ and each subset $X \subset V(T)$, let 
\[\defmath{w(T,X)} = \sum_{v \in X} q^{\depth_T(v)}.\]
We also write $\defmath{w(T)} = w(T,V(T))$.

\begin{lemma}\label{lem:weight}
Let $T$ be a rooted binary tree and let $k$ be a nonnegative integer.
If $T$ is isomorphic to a subgraph of $H\boxtimes P$, where $P$ is a path and $\pw(H)\le k$, then
\[
    w(T)
    \le 388\cdot 583^k.
\]
\end{lemma}

\begin{proof}
%Set $q = 3/4$.
The proof is by induction on $k$.
Assume that $T$ is isomorphic to a subgraph of $H\boxtimes P$ where $P$ is a path and $\pw(H)\le k$.
Say that each $v \in V(T)$ is mapped by this isomorphism to $(\pi(v),\lambda(v)) \in V(H) \times V(P)$.
Restricting $H$ to the subgraph induced by the first coordinates used by $T$, we may assume that $H$ is connected.
% For every edge $uv\in E(T)$,
% \[
%     \dist_H(\pi(u),\pi(v))\le 1
%     \ \ \text{ and } \ \
%     |\lambda(u)-\lambda(v)|\le 1.
% \]

If $k=0$, then $H$ has one vertex, so $T$ is a subgraph of a path.
There are at most two vertices of $T$ at each positive depth, and therefore
\[
    w(T)\le 1+2\sum_{d=1}^\infty q^d=7\le 388.
\]

Now let $k\ge 1$, let $r$ be the root of $T$, and apply \cref{lem:separator} to $H$ with $z=\pi(r)$.
Let $Y$ be the resulting set and let
\[
    S=\{v\in V(T):\pi(v)\in Y\}.
\]
Since $z=\pi(r)\in Y$, we have $r \in S$.
Note that for $v \in S$, if $d = \depth_T(v)$, then 
$\dist_H(\pi(r),\pi(v)) \leq d$ and $\dist_P(\lambda(r),\lambda(v)) \leq d$. Equivalently, $\pi(v) \in Y \cap B_H(z,d)$ and $\lambda(v) \in B_P(\lambda(r),d)$.
We have $|Y \cap B_H(z,d)| \leq 2d+1$ and $|B_P(\lambda(r),d)| \leq 2d+1$.
In particular, there are at most $2d+1$ choices for each coordinate, so injectivity of the embedding gives
\[
    |\{v\in S:\depth_T(v)=d\}|\le (2d+1)^2.
\]
Consequently,
\[
    w(T,S)
    \le \sum_{d=0}^\infty (2d+1)^2q^d
    =\frac{1+6q+q^2}{(1-q)^3}
    =388.
\]

Let $\calC$ be the set of components of $T - S$. 
For every $C \in \mathcal{C}$, we proceed as follows. 
Let $r_C$ be the vertex in $C$ of minimum depth in $T$.
We consider $C$ to be rooted at $r_C$.
Since $r \in S$, we have $r_C \neq r$, and so  the parent of $r_C$ in $T$ is in $S$.
Moreover, for every $v \in V(C)$, we have $\depth_T(v) = \depth_T(r_C) + \depth_C(v)$.
In particular,
\begin{equation}
w(T,V(C)) = q^{\depth_T(r_C)} \cdot w(C).
\label{eq:w(T,C)}
\end{equation}
%Since a connected subgraph of a tree contains the path between each pair of its vertices, $C$ has a unique vertex $c$ of minimum depth. %: otherwise the least common ancestor of two such vertices would be a shallower vertex of $C$.
%The parent $s$ of $c$ belongs to $S$, and the same argument shows that $c$ is an ancestor of every vertex of $C$.
%Thus $C$, rooted at $c$, is again a binary tree.
The first coordinates used by $C$ induce a connected subgraph of $H-Y$. This subgraph has pathwidth at most $k-1$. 
By induction,
\begin{equation}
    w(C) \leq 388\cdot 583^{k-1}.
\label{eq:induction-call}
\end{equation}
For every $v \in S$, we define
\[Z_v = \{C \in \calC : \text{$r_C$ is a child of $v$}\}.\]
Since $T$ is a binary tree, for every $v \in S$, we have $|Z_v| \leq 2$.
As discussed before, 
\begin{equation}\textstyle
\bigcup_{v \in S} Z_v = \calC. 
\label{eq:partition-of-C}
\end{equation}
Therefore,
\begin{align*}
    w(T,V(T) \setminus S) &= 
    \sum_{C\in\mathcal{C}} w(T,V(C))\\
    &= 
    \sum_{C\in\mathcal{C}} q^{\depth_T(r_C)}\cdot w(C)&&\textrm{by~\eqref{eq:w(T,C)}}\\
    &= \sum_{v \in S} \sum_{C \in Z_v} q^{\depth_T(v)+1} \cdot w(C)
    &&\textrm{by~\eqref{eq:partition-of-C}}\\
    &\leq \sum_{v \in S} 2q \cdot q^{\depth_T(v)} \cdot (388\cdot 583^{k-1}) &&\textrm{by~\eqref{eq:induction-call}}\\
    &\leq 2q \cdot w(T,S) \cdot (388\cdot 583^{k-1}) \leq 582\cdot 583^{k-1} \cdot w(T,S).
\end{align*}
Summarizing, we have
\begin{align*}
    w(T) &= w(T,V(T) \setminus S) + w(T,S) \leq w(T,S) \cdot (582\cdot 583^{k-1} + 1) \le 388 \cdot 583^k. 
\end{align*}
This completes the proof of the lemma.
\end{proof}

\begin{proof}[Proof of \cref{thm:main}]
Let $h$ be a nonnegative integer and let $k = \rpw(T_h)$.
Choose $H$ with $\pw(H)=k$ and a path $P$ such that $T_h$ is isomorphic to a subgraph of $H\boxtimes P$.
The tree $T_h$ has $2^h$ leaves, all at depth $h$, hence
\[w(T_h) \geq 2^h \cdot q^{h} = (3/2)^h.\]
By \cref{lem:weight},
\[
    (3/2)^h \leq w(T_h) \le 388\cdot 583^k.
\]
Since
\[
    388<(3/2)^{15}
    \ \ \text{ and } \ \
    583< (3/2)^{16},
\]
we obtain $h<15+16k$, and therefore $k\ge \left\lfloor(h+1)/16\right\rfloor$.
% For completeness, let $p_h:=\pw(T_h)$.  For $h\ge 2$, choose a leaf-to-leaf path $Q$ through the root.  Every component of $T_h-V(Q)$ is a complete binary tree of height at most $h-2$.  Take path-decompositions of these components of width at most $p_{h-2}$, add their attachment vertex in $Q$ to every corresponding bag, concatenate the decompositions in the order of their attachment vertices along $Q$, and insert a two-vertex bag for every edge of $Q$.  This gives
% \[
%     p_h\le p_{h-2}+1.
% \]
% Since $p_0=0$ and $p_1=1$, we have $p_h\le\lceil h/2\rceil$.  Finally, $\rpw(G)\le\pw(G)$ for every graph $G$, by placing all vertices in one $P_\infty$-coordinate.  Hence
% \[
%     \rpw(T_h)\le\left\lceil\frac h2\right\rceil.
% \]
\end{proof}

\section*{Statement of AI use}
The proof was found by OpenAI's GPT-5.6 Sol Pro. 
The authors take responsibility for the
mathematical correctness of the presented arguments.

% ----------------------------
\bibliographystyle{abbrv}
\bibliography{bibliography}
\end{document}